\documentclass[12pt]{amsart}

\usepackage{amsmath}
\usepackage[english,  activeacute]{babel}
\usepackage[latin1]{inputenc}
\usepackage{amssymb}
\usepackage{amsthm}
\usepackage{graphics,graphicx}
\usepackage{enumerate}
\usepackage{array}
\usepackage{a4wide}
\usepackage{color, url}
\usepackage{float}
\usepackage{bm}
\usepackage{cite}

\theoremstyle{plain}

\newtheorem{theorem}{Theorem}[section]

\newtheorem{corollary}[theorem]{Corollary}
\newtheorem{proposition}[theorem]{Proposition}

\newcommand{\NC}{{\texttt{NC}}}

\newcommand{\mP}{{\mathcal P}}

\makeatletter
\renewcommand*\env@matrix[1][*\c@MaxMatrixCols c]{%
  \hskip -\arraycolsep
  \let\@ifnextchar\new@ifnextchar
  \array{#1}}
\makeatother

\title[Restricted Dyck Paths, Polyominoes, and Non-Crossing Partitions]
{Some Connections  Between Restricted Dyck Paths, Polyominoes, and Non-Crossing Partitions} 
\author{Rigoberto Fl\'orez}\thanks{}
\address{\noindent Department of Mathematical Sciences,  The Citadel,  Charleston, SC, U.S.A.}
\email{rigo.florez@citadel.edu}

\author{Jos\'e L. Ram\'{\i}rez}
\address{\noindent Departamento de Matem\'aticas, Universidad Nacional de Colombia, Bogot\'a,  COLOMBIA}
\email{jlramirezr@unal.edu.co}

\author{Fabio A. Velandia}
\address{\noindent Departamento de Matem\'aticas, Universidad Nacional de Colombia, Bogot\'a,  COLOMBIA}
\email{fvelandias@unal.edu.co}

\author{Diego Villamizar}
\address{Escuela de Ciencias Exactas e Ingenier\'ia\\Universidad Sergio Arboleda\\Bogot\'a, Colombia}
\email{diego.villamizarr@correo.usa.edu.co}

\date{\today}
\subjclass[2010]{05A15; Secondary 05B50.}
\keywords{Dyck path, polyomino, partition, bijection, generating function}

\begin{document}
\begin{abstract}
\noindent 
A \emph{Dyck path} is a lattice path in the first quadrant of the $xy$-plane that starts at the origin, ends on the $x$-axis, and consists of the same number of  
North-East steps  $U$ and South-East steps  $D$. A \emph{valley} is a subpath of the form $DU$.  A Dyck  path is called  \emph{restricted $d$-Dyck} 
if the difference between any two consecutive valleys is at least $d$  (right-hand side minus left-hand side) or if it has at most one valley. In this paper we give some 
connections between restricted $d$-Dyck paths and both, the non-crossing partitions of $[n]$ and some subfamilies of polyominoes. 
We also give generating functions to count several aspects of these combinatorial objects. Accepted for publication in Proceedings of the 52nd Southeastern International Conference on Combinatorics, Graph Theory, and Computing.
\end{abstract}

\maketitle

\section{Introduction}

A \emph{Dyck path} is a lattice path in the first quadrant of the $xy$-plane that starts at the origin, ends on the $x$-axis, and consists of (the same number of)  
North-East steps  $U=(1,1)$ and South-East steps  $D=(1,-1)$. A \emph{peak} is a subpath of the form $UD$, and a \emph{valley} is a subpath of the form $DU$. 
We define the {\em valley vertex} of $DU$ to be the lowest point (a local minimum) of $DU$. Following \cite{FloRamVelVil}, we define the vector  
$\nu=(\nu_1, \nu_2, \dots, \nu_k)$, called \emph{valley vertices},  formed by all $y$-coordinates (listed from left to right) of all valley vertices of a Dyck path.  
For a fixed $d \in \mathbb{Z}$, we introduce a new family of   lattice  paths called   \emph{restricted $d$-Dyck} or  \emph{$d$-Dyck}  
(for simplicity).  Namely, a Dyck path $P$ is a \emph{$d$-Dyck},  if either $P$   
has at most one valley, or  if its valley vertex vector $\nu$ satisfies that  $\nu_{i+1}-\nu_i\geq d$, where   $1\le i <k$. For example,  in Figure \ref{Example}  
we see that $\nu=(0,2,5,7)$, and that $\nu_{i+1}-\nu_i\ge 2$, for $ 1\le i< 4$. So, the figure depicts a $2$-Dyck path of length 28 (or semi-length 14).  Another 
well-known example, is the set of $0$-Dyck paths, known in literature as non-decreasing Dyck paths (see for example, \cite{barcucci,CZ,CZ2,ElizaldeFlorezRamirez,FlorezJose}). 
A second classic  example occurs when $d\to -\infty$, giving rise to Dyck paths. 
We say that a polyomino $P$  is \emph{directed} if for a given cell  $c\in P$ there is a path totally contained in $P$ joining $c$ with the bottom left-hand corner cell and using only 
north and east steps. We say that $P$ is \emph{column-convex} if every vertical path joining the bottom cell with the top cell in the same column is fully contained in $P$. 
A polyomino that is both directed and column-convex is denoted by dccp \cite{barcucci2}.

\begin{figure} [htbp]
\begin{center}
\includegraphics[scale=0.9]{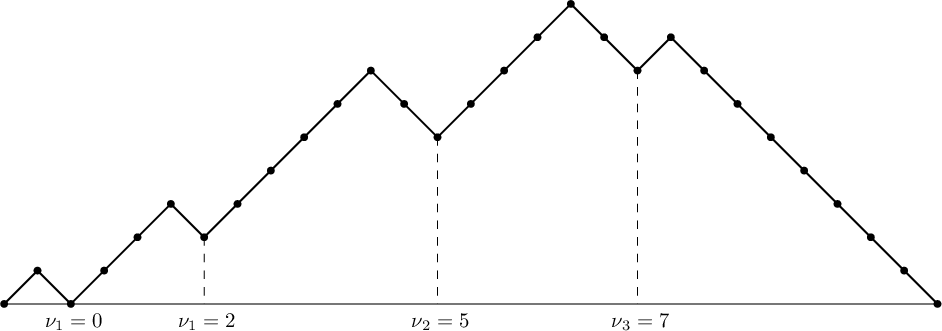}
\end{center}
\caption{A $2$-Dyck path of length 28.} \label{Example}
\end{figure}

Deutsch and Prodinger \cite{DeutscheProdinger}  give a bijection between the  set of non-decreasing Dyck paths of  length $2n$ and the set of directed column-convex 
polyominoes (dccp). For $d\ge 0$, we say that a dccp is $d$-\emph{restricted} if either it is formed by exactly one or two columns or if the difference between any two of its 
consecutive initial altitudes is at least $d$  (right-hand side minus left-hand side, but not including the first initial altitude). The left-hand side of Figure \ref{ipldefi} depicts a $2$-polyomino,  
where  the initial altitudes are $(0,0,2,4)$. In this paper we give a combinatorial expression and a generating function to count the number of $d$-restricted polyominoes 
of area $n$. When $d=0$, we obtain the result given by Deutsch and Prodinger.

\begin{figure}[h]
\centering
\includegraphics[scale=0.35]{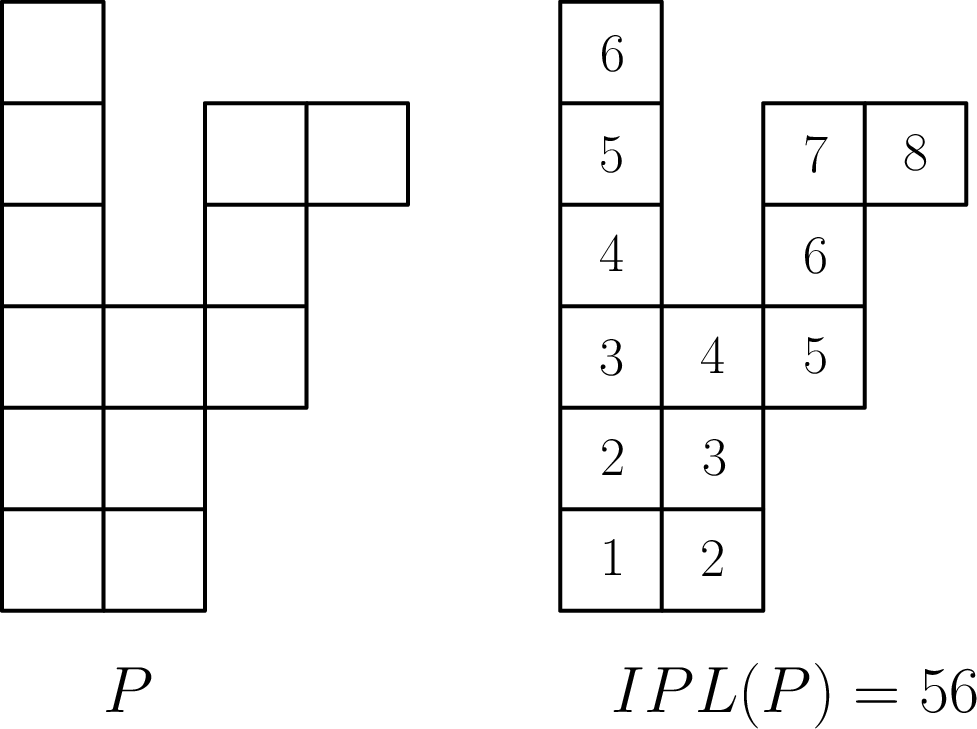}
\caption{A $2$-polyomino $P$ with TIPL equal to $56$.}
\label{ipldefi}
\end{figure}

The \textit{internal path length} of a cell in a dccp is the minimum number of steps needed to reach the cell, starting at the origin of the dccp and moving from one cell to 
any one of the two adjacent cells. The \textit{total internal path length} of dccp is the sum of the internal path length over the set of its cells. In this paper 
we give a generating function to count the total internal path length (TIPL) when the polyominoes are in the family of the $d$-restricted. 
The left-hand side of Figure \ref{ipldefi} shows the polyomino, while the right-hand side shows the internal path lengths of each cell, from with the total internal path length is seen to be $56$.

A fixed partition $P$ of $[n]$ is called \emph{non-crossing} if every edge of the form $\{n_1,n_2\}\subset [n]$ of its associated graph (defined in Section \ref{RestrictedNonCrossingPartitions}) connecting two distinct elements of $P$   
do not cross each other.  In this paper we extend this concept (see restricted $d$-partitions of $[n]$ on Page \pageref{dBell}) and give a connection between the $d$-Dyck  
paths and the non-crossing partitions of $[n]$.

\section{A Connection with the Polyominoes}\label{ConnectionPolyominoes}

The \emph{area} of a polyomino is the number of its cells.    
The right-hand side of Figure \ref{Example4} shows a dccp of area 12.  The entries  of the vector $A=(0, 0, 2, 5)$ represent the initial altitude (height) of each column of the polyomino 
and the entries of the vector $B=(2, 4, 7, 6)$  represent the final altitude (height) of each  column of the polyomino.

Deutsch and Prodinger \cite{DeutscheProdinger}  give a bijection between the  set of non-decreasing Dyck paths of   length $2n$ and the set of dccp of area $n$.    
The bijection from \cite{DeutscheProdinger} can be described as follows. Let $A=(0, a_2, \dots, a_k)$ and $B= (b_1, b_2, \dots, b_k)$ be vectors  
formed by the initial altitudes and final altitudes, from left to right, of the columns of a dccp. 
If $D$ is a non-decreasing Dyck path with valley vertices vector $V=(a_2, \dots, a_k)$ and peak vertices vector $P^{\prime}=(b_1, b_2, \dots, b_k)$ from left to right,   
then $A$ and $B$ are bijectively related with  $V$ and $P^{\prime}$, respectively.    
For example, the dccp in the right-hand side of Figure \ref{Example4} maps bijectively to the  Dyck path (on the left-hand side). Note that the valley vertices 
vector  and  the peak vertices vector of the path in Figure \ref{Example4} are  $V=(0, 2, 5)$ and $P^{\prime}=(2, 4, 7, 6)$, respectively. Clearly these two vectors, $V$ and $P^{\prime}$,
are bijectively related with $A=(0, 0, 2, 5)$ and $B=(2, 4, 7, 6)$.

\begin{figure} [htbp]
\centering
\includegraphics[scale=.65]{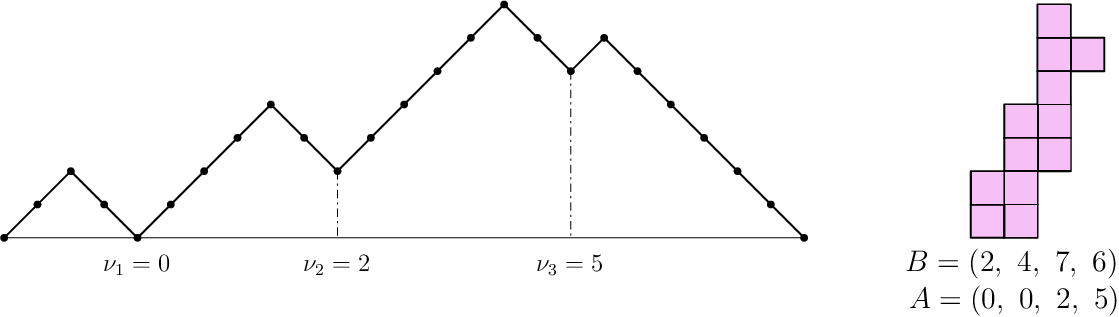}
\caption{A bijection between dccp polyomino and non-decreasing Dyck path.} \label{Example4}
\end{figure}

We say that a dccp is $d$-\emph{restricted} for $d\ge 0$, if either it is formed by exactly one or two columns or if its initial altitudes vector $A= (0, a_2, \dots, a_k)$ satisfies  that  
$a_{i+1}-a_i\geq d$  for all $2\leq i\leq k-1$. In Figure \ref{Example3} we show all $2$-restricted polyominoes of area 7, with exactly 3 columns.   

\begin{figure} [H]
\begin{center}
\includegraphics[scale=0.7]{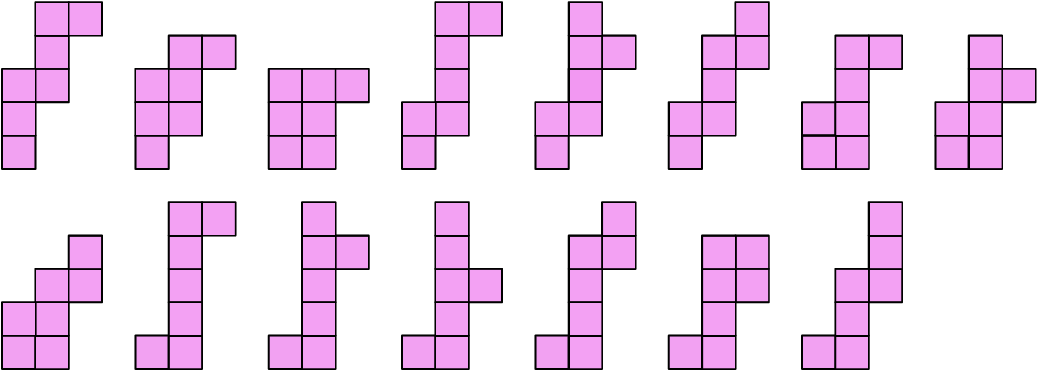}
\end{center}
\caption{All $2$-restricted polyominoes of area 7 with 3 columns. } \label{Example3}
\end{figure}

From \cite{FloRamVelVil} we know that  the number of $d$-Dyck paths of semi-length $n$ for  $d\geq 1$ is given by 
\begin{equation}  \label{FloRamVelVilrdn}
\sum_{k=0}^{\lfloor\frac{n+d-2}{d}\rfloor}\binom{n-(d-1)(k-1)}{2k}.
\end{equation}

From the bijection described in the second paragraph of this section we conclude that the set of $d$-Dyck paths of semi-length $n$ are bijectively related with the set of the $d$-polyominoes of area $n$. Therefore, we have this result. 
 
 \begin{proposition} The expression on \eqref{FloRamVelVilrdn} also counts the total number of $d$-restricted polyominoes  of area $n$.
 \end{proposition}
 
 We use  $\mP _d(n)$ to denote the family of $d$-restricted polyominoes of length $n$ ($n$ columns). Associated to this concept we define these three sets.
 $$\mP_d^*(n)=\{P\in \mP_d(n): a_{i+1}-a_{i}\geq d, \text{ for all } i\geq 1 \},$$ 
$$\mP_d=\bigcup_{n\geq 0}\mP_d(n),\quad  \text{ and  } \quad \mP_d^*=\bigcup_{n\geq 0}\mP_d^*(n).$$

 Notice that the elements in $\mP_d^*(n)$ satisfy that the difference between the initial altitude of the second column and the first column is at least $d$.    
 
\begin{theorem} Let  $v_d(n)$ be the number of $d$-restricted polyominoes of area $n$. 
Then the generating function of the sequence $v_d(n)$ is given by
\begin{equation} \label{expression2RF}
V_d(x)=\sum_{n\geq 0}v_d(n)x^n=\frac{1-2x+2x^2-x^{d+1}}{(1-x)(1-2x+x^2-x^{d+1})}.
\end{equation}
\end{theorem}

\begin{proof}
For any $d$-restricted polyomino $P\in \mP_d$, then $P$ is either a (possibly empty) column  or can be obtained by ``gluing" a polyomino $P^*$ in $\mP^*_d$  to the right-hand side of column $C$. That is, the lowest level of $P^*$ must be at the same level of a chosen cell $c_i$ in $C$  as shown on the right-hand side of Figure  \ref{decomposition}. 
\begin{figure}[H]
\centering
\includegraphics[scale=1]{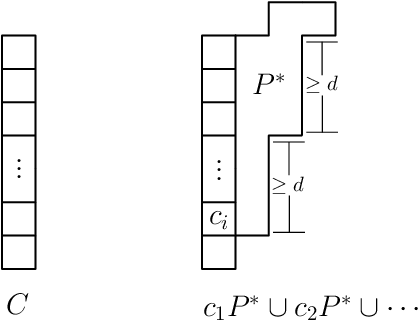}
\caption{Decomposition of a $d$-restricted polyomino.}
\label{decomposition}
\end{figure}
Let $V_d^*(x)$ be the generating function of the area of the polyominoes in $\mP^*_d$.  Let $C(x)$ be the generating function of the area of the nonempty columns (polyominoes with only one column). Notice that if the column has height $n$, then this case contributes to the generating function the term $x^n$. Thus, 
$$1+C(x)=1+\sum_{n\geq 1}x^n=1+\frac{x}{1-x}=\frac{1}{1-x}.$$  
So, $\partial(C(x))/\partial x =\frac{x}{(1-x)^2}$ is the area generating function of the nonempty single column polyominoes 
with a labeled cell.  
From the symbolic method (see \cite{flajolet}), we obtain the functional equation
\begin{align*}
V_d(x)=\underbrace{(1+C(x))}_{(1)}+ \underbrace{\dfrac{\partial(C(x))}{\partial x} V_d^*(x)}_{(2)}= \frac{1}{1-x} + \frac{x}{(1-x)^2}V^*_d(x).
\end{align*}
The expression \eqref{FloRamVelVilrdn} in the above equality  corresponds to the generating function for the area of the columns (possible empty). 
The expression \eqref{expression2RF}  is the generating function for the $d$-restricted polyominoes with at least two columns such that the polyomino 
$\mP^*$ that starts in column 2 is in $\mP^*_d$.  In order to find $V^*_d(x)$ we can apply a similar decomposition to the family $\mP^*_d$ as shown in  
Figure \ref{decomposition2}. Notice that any $P^{*}\in \mP^*_d$ with at least one column can be obtained by attaching a polyomino $P'\in \mP_d^*$ to any  
but the first $d$ cells of a column of area greater than $d$  (this is necessary to preserve the inequalities on the initial altitudes vector). Thus 
\begin{align*}
V_d^*(x)= \frac{1}{1-x} + x^d\cdot \frac{x}{(1-x)^2}V^*_d(x),
\end{align*}
this implies that 
\begin{equation*}
V^*_d(x)=\frac{1-x}{1-2x+x^2-x^{d+1}}.
\end{equation*}

Therefore, we obtain the desired result.
\end{proof}

\begin{figure}[H]
\centering
\includegraphics[scale=0.3]{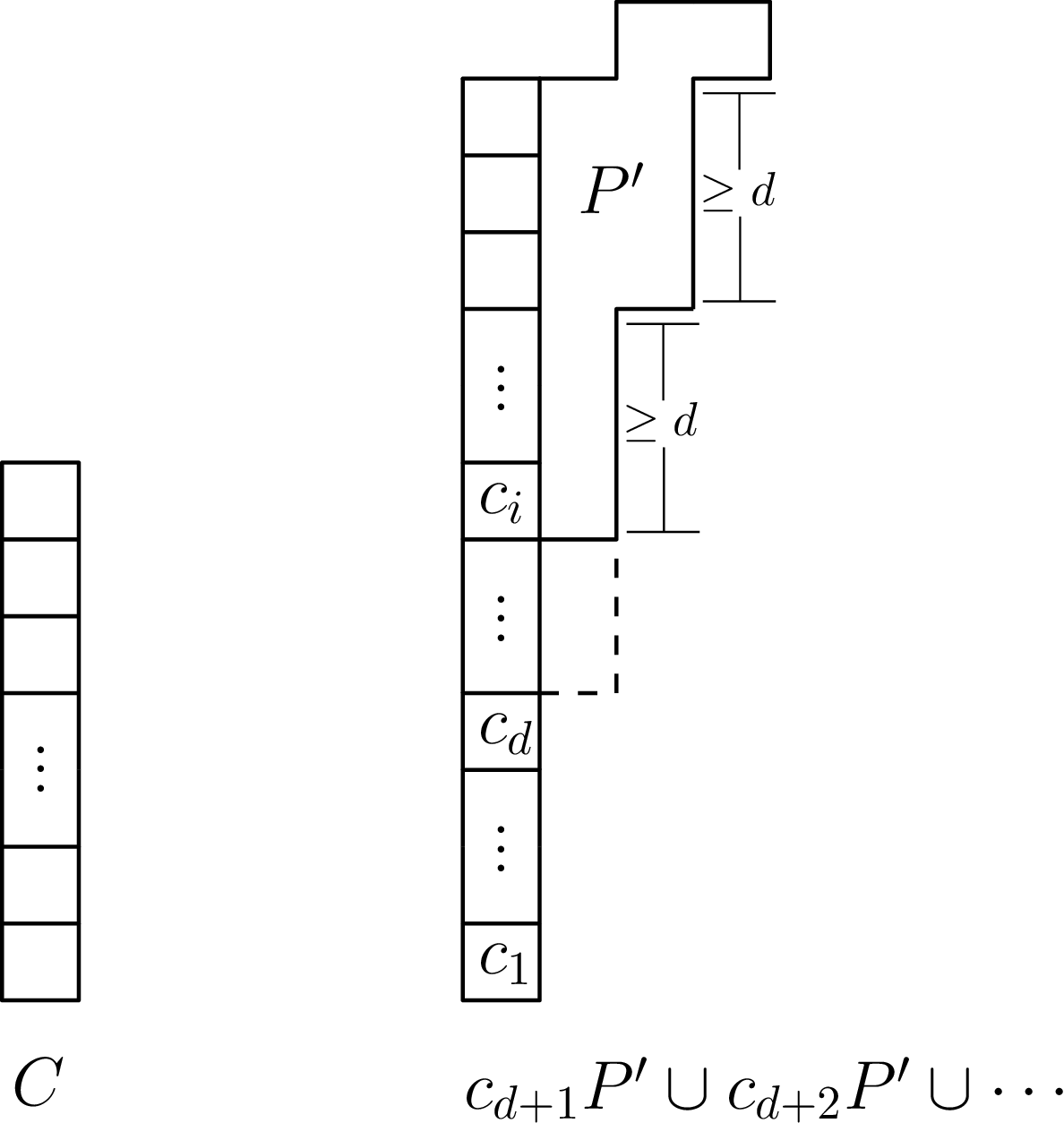}
\caption{Decomposition of the polyominoes in $\mP^*$.}
\label{decomposition2}
\end{figure}

\subsection{Total path length of the polyominoes}
Let $P$ be a dccp. The \textit{internal path length} (IPL) of a cell $c$ in $P$ is the minimum number of steps needed to reach 
$c$ starting at the origin (the bottom leftmost cell)  of $P$ and moving from one cell to any one of the two adjacent cells.  
The \textit{total internal path length} of $P$  (TIPL) is defined to be the sum of the IPL over the set of its cells. For example, 
Figure \ref{ipldefi} shows a $2$-restricted polyomino such that each cell is labeled  with the  minimum number of steps required 
to walk from the origin. So, the TIPL of this polyomino is 56 (it can be obtained  adding by up all of these labels). We give a generating   
function to count the TIPL. This result generalizes the result given by Barcucci et al. \cite{barcucci2} for the case $d=0$.

We use $t_d(n)$ to denote the total internal path length of all the $d$-polyominoes of area $n$. The following theorem gives a generating function for the sequence $t_d(n)$.
\begin{theorem}\label{IPLTheo}
For $d\geq 0$, we have the rational generating function
\begin{equation*}
T_d(x):=\sum_{n\geq 0}t_d(n)x^n=\frac{f_d(x)}{g_d(x)},
\end{equation*}
where $g_d(x)=2((x-1)^3-(x-1)x^{d+1})^3$ and 
\begin{multline*}
f_d(x)=-\left(d^2+7 d+18\right) x^{d+3}+2 \left(2 d^2+10 d+13\right) x^{d+4}-2 \left(3 d^2+9 d+11\right) x^{d+5}\\
+4 \left(d^2+d+2\right) x^{d+6}-\left(d^2-3 d-10\right) x^{2 d+4}+\left(2 d^2-4 d-6\right) x^{2 d+5}+6 x^{d+2}\\
-(d-1) d x^{d+7}-6 x^{2 d+3}-(d-1) d x^{2 d+6}+2 x^{3 d+4}-6 x^7+22 x^6-30 x^5+20 x^4-10 x^3+6 x^2-2x.
\end{multline*}

\end{theorem}
\begin{proof}
Let $T_d^*(x)$ be the generating function of  the TIPL of all $d$-restricted polyominoes of area $n$ in $\mP_d^*$. We use again the decomposition given in 
Figure \ref{decomposition}. Since the TIPL of a single column with $n$ cells is $n(n+1)/2$, we have 
\begin{align*}
T_d(x)=\sum_{n\geq 0} \frac{n(n+1)}{2}x^n + Q_d(x)= \frac{x}{(1-x)^3}+Q_d(x),
\end{align*}
where $Q_d(x)$ is the generating function of the TIPL of $d$-restricted polyominoes with at least two columns. According to the  decomposition given in 
Figure \ref{decomposition}, the TIPL contribution in $Q_d(t)$ can be divided into three parts:

\begin{enumerate}[\textbf{Part} (1).]

\item \emph{The TIPL contribution of the family $\mP_d^*$  to the right of the first column}. Given that the whole family of $d$-restricted  polyominoes in $\mP_d^*$ can 
be attached to a particular cell, the corresponding generating function is given by
\begin{equation} \label{FabionExp}
\left(\sum_{n\geq 1}nx^n \right)T_d^*(x)=\frac{x}{(1-x)^2}T_d^*(x).
\end{equation}
That is,  the expression  \eqref{FabionExp} is the product of the generating function of the cells in a column (without contributions to the TIPL) and the generating function of 
the TIPL in the family $\mP_d^*$.

\item  \emph{The TIPL contribution of the first column}. In this case, the TIPL of a column equals the contributions of the smaller columns ending at a cell having a $d$-restricted  
polyomino in $\mP_d^*$ is attached. In order to distinguish cells, the generating function $\mathcal{S}:=x\frac{d}{dx}\sum_{m\ge 0} {m+d+1 \choose 2}x^n$ must be considered. This TIPL contribution must  
be considered for each non-empty  polyomino  in $\mP_d^*$  that is attached. Thus, the generating function of the TIPL contributed by the first column is given by 

\begin{equation*}
\mathcal{S}\cdot(V_d^*(x)-1)=\mathcal{S}\cdot V_d^*(x)-\mathcal{S}.
\end{equation*}

\item  \emph{The TIPL contribution of the cells of $\mP_d^*$, relative to the origin on the first column}. Similarly, $x\frac{d}{dx}V_d(x)$ is the generating function of the cells in   
$\mP_d^*$ (the altitude of this cell is greater than or equal to $d$). However, the generating function associated to the TIPL contribution of every cell must be modified, not only to exclude attachments to the first $d$ cells but   
also to add the increase on $d$ units to the TIPL. The TIPL of the upper cells, with the origin placed on the $(d+1)$-th cell, is represented by 

\begin{equation*}
x^d\sum_{m\geq 0}\binom{m+1}{2}x^m=x^d\cdot \frac{x}{(1-x)^3}.
\end{equation*}
For each one of the upper cells, their IPL relative to the column of the origin, equals the previously described IPL increased by $d$. In terms of generating functions,   
the increase in $d$ units for every upper cell is represented by
\begin{equation*}
dx^d\sum_{m\geq 0}m x^m=dx^d\cdot \frac{x}{(1-x)^2}.
\end{equation*} 
Thus, the TIPL contribution of every upper cell is represented by the sum
\begin{equation*}
x^d\cdot \frac{x}{(1-x)^3}+dx^d\cdot \frac{x}{(1-x)^2}=x^d\left(\frac{x}{(1-x)^3}+d \frac{x}{(1-x)^2}\right).
\end{equation*} 
Therefore,  the generating function representing the TIPL contribution of the first column is
\begin{equation*}
x^d\left(\frac{x}{(1-x)^3}+d \frac{x}{(1-x)^2}\right)x\frac{d}{dx}V_d^*(x).
\end{equation*}
\end{enumerate}
From the previous analysis we have
\begin{align*}
T_d^*(x)&=\frac{x}{(1-x)^3}+Q_d^*(x) \\
&=\frac{x}{(1-x)^3}+\frac{x^{d+1}}{(1-x)^2}T_d^*(x) \\ 
&+x^d\left(\sum_{m\geq 0}m\binom{m+d+1}{2}x^m\right)V_d^*(x)-x^d\sum_{m\geq 0}m\binom{m+d+1}{2}x^m \\
&+x^d\left(\frac{x}{(1-x)^3}+d \frac{x}{(1-x)^2}\right)x\frac{d}{dx}V_d^*(x)\\
&=\frac{x}{(1-x)^3}+\frac{x^{d+1}}{(1-x)^2}T_d^*(x) \\ 
&+x^d(V_d^*(x)-1)\cdot x \frac{d}{dx}\left(\sum_{m\geq 0}m\binom{m+d+1}{2}x^m\right)\\
&+x^d\left(\frac{x}{(1-x)^3}+d \frac{x}{(1-x)^2}\right)x\frac{d}{dx}V_d^*(x).
\end{align*}
Solving for $T_d^*(x)$ and afterwards for $T_d(x)$, the result follows.

\end{proof}

In particular for $d=0$ we recover the following result for the directed column-convex 
polyominoes.

\begin{corollary}[\cite{barcucci2}, Theorem 4.1]
The generating functions for the TIPL of  the directed column-convex 
polyominoes is given by
$$T_0(x)=\frac{x \left(3 x^4-9 x^3+8 x^2-4 x+1\right)}{(1-x)\left(x^2-3 x+1\right)^3}.$$
\end{corollary}

For example, the series expansion of $T_2(x)$ is 
\begin{align*}
T_2(x)&=\frac{x \left(4 x^6-14 x^5+15 x^4-8 x^3+2 x^2-x+1\right)}{(x-1) \left(x^3-x^2+2
   x-1\right)^3}\\ 
&=x+6 x^2+23 x^3+65 x^4+\bm{165} x^5+401 x^6+932 x^7+2081 x^8+4516 x^9+\cdots .
\end{align*} 
Thus, the TIPL of all 2-restricted polyominoes of area $5$ is equal to $165$.  Figure \ref{iplex} shows both the IPL and the TIPL  of each 2-restricted polyomino  of area 5.
\begin{figure}[h]
\includegraphics[scale=0.6]{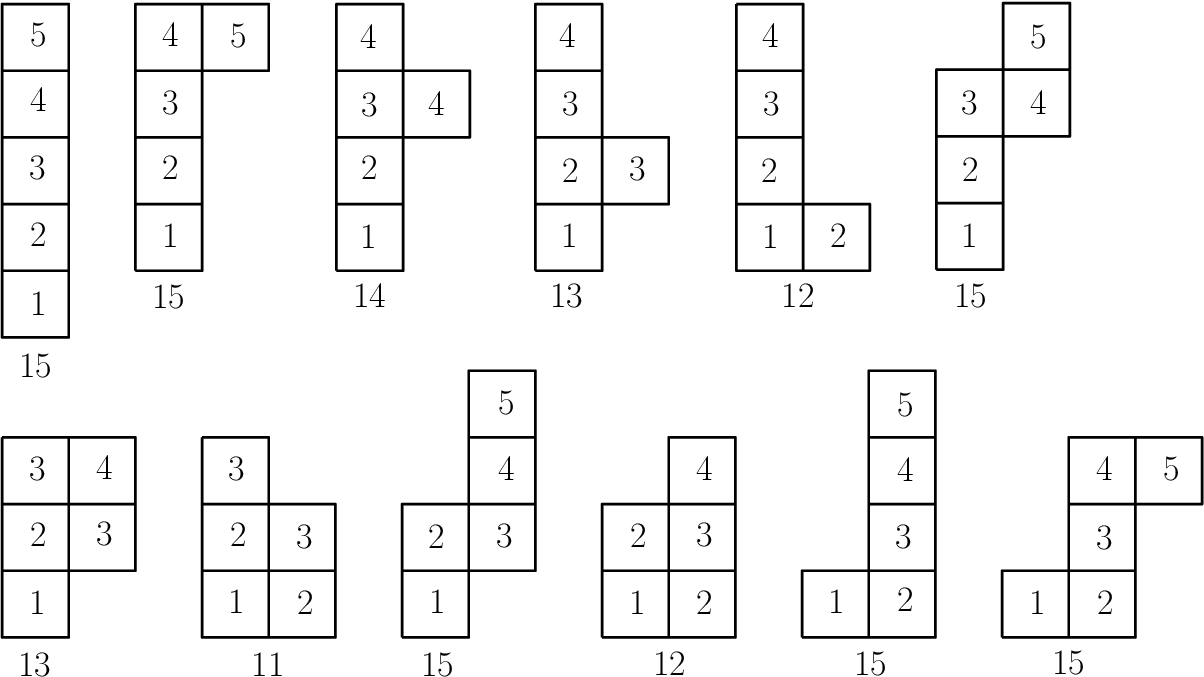}
\caption{The TIPL and IPL of each $2$-restricted polyomino of area 5.}
\label{iplex}
\end{figure}

\section{Restricted Non-Crossing Partitions} \label{RestrictedNonCrossingPartitions}

In this section, we describe a connection between the $d$-Dyck paths and the non-crossing partitions.  Before doing so, let us recall some   
terminology and make a few definitions.  A \emph{partition} of $[n]:=\{1, 2, \dots, n \}$ is a collection of mutually disjoint non-empty sets   
whose union is $[n]$. An element of the partition is called a \emph{block}.  The cardinality of the set of partitions of $[n]$ having exactly $k$   
blocks is given by  the Stirling number of the second kind ${n \brace k}$. The set of all partitions of $[n]$ is enumerated by 
$B(n)=\sum_{k=0}^n{n \brace k}$, the $n$-th Bell number. For  $n, k \geq 0$, we use $\Pi(n,k)$ to denote the set of all partitions of $[n]$ having $k$  
blocks, and use $\Pi(n)$ to denote $\cup_{k=0}^n\Pi(n,k)$.  For example, $B(3)=5$, with the corresponding partitions being
\begin{align*}
&\{\{ 1, 2, 3\}\},& &\{\{ 1, 2 \},  \{3\}\}, && \{\{ 1, 3 \},  \{ 2 \}\}, && \{\{ 1 \},  \{2,3 \}\},& &\{\{  1\},  \{2 \}, \{ 3\}\}.&
\end{align*}

Suppose that $\pi$ in $\Pi(n,k)$ is represented as $\pi=B_1/B_2/\cdots/B_k$, where $B_i$ is a block of $\pi$, for $1\leq i \leq k$.  
(Note that different blocks are separated by the symbol $/$.) The graph on the vertex set $[n]$ whose edge set consists of arcs connecting the elements of   
each block in numerical order is called the \emph{graph representation} of $\pi$. For example, in Figure \ref{grafo1} we depict the graph representation of the set partition
 $\pi=\{\{1, 4\}, \{2, 6, 7\}, \{3\}, \{5, 8\}\}\in \Pi(8,4)$.
 
\begin{figure}[h]
    \centering
    \includegraphics[scale=0.8]{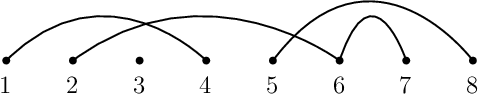}
    \caption{Graph representation of $\pi=\{\{1, 4\}, \{2, 6, 7\}, \{3\}, \{5, 8\}\}$.}
    \label{grafo1}
\end{figure}

A set partition is called \emph{non-crossing} if none of the edges on the graph representation cross ---in the graph representation. 
Let $\NC(n)$ denote the set of non-crossing set partitions of $[n]$. It is well-known that $|\NC(n)|=C_n$, where $C_n=\frac{1}{n+1}\binom{2n}{n}$ is the $n$th   
Catalan number.  Here we sketch the bijection between $\NC(n)$ and the Dyck paths of semi-length $n$.

Let $P$ be  a Dyck path of semi-length $n$. This path can be represented as a word over the alphabet $U$ and $D$. We use $U$ to denote a North-East step 
$(1,1)$ and use $D$  to denote the South-East step $(1,-1)$. Therefore, any Dyck path can be written as  
$U^{a_1}D^{b_1}\cdots U^{a_n}D^{b_n}$, where $a_i,b_i\geq 0$ (factor it in such a way that if $a_i=0$, then $b_j=a_j=0$ for $j>i$), 
$\sum_{j\ge 1}^{i} b_j\le \sum_{j\ge 1}^{i}  a_j$ for every $1\le i \le n$, and if $i=n$ the equality holds. Enumerate, starting with $1$, 
all $U$ steps. Notice that if we write the Dyck path $P$ as $P=P_1P_2\cdots P_{2n-1}P_{2n}$, with $P_i\in \{U,D\}$, then every 
$D$ step, say $P_j=D$, on $P$ has a corresponding $U$ step, say $P_{j'}=U$, such that $j'<j$ and either $j'=j-1$ or 
$\widetilde{P}=P_{j'+1}P_{j'+2}\cdots P_{j-1}$ is a Dyck path.
Now, for every $1\leq j\leq n$ take $b_j\ne0$ consecutive  $D$ steps, match them with their corresponding $U$ steps, and then form a 
block with the subscripts of the (corresponding) $U$ labels. At the end of this procedure we obtain a partition of $[n]$. 
It is known that this partition is non-crossing (see, for example, \cite{Zhao}).   
This process can be inverted and it is a bijection.  
We denote this bijection by $\Phi$.   For example, for the Dyck path $P$ in Figure \ref{NCfig2}, we have that $P=U^4D^2U^5DUD^7$. 
That is, $b_1=2, b_2=1$, and $b_3=7$.  The first two South-East steps, $D^2$, on the left-hand side correspond to the labels $3$ 
and $4$. The next South-East step corresponds, $D$,  to the label $9$, 
and finally the last seven South-East steps correspond,  $D^7$,  to the labels $1, 2, 5, 6, 7, 8$, and $10$. Therefore,
$\Phi(P)=\{\{1,2,5,6,7,8,10\}, \{3, 4\}, \{ 9\} \}$. 

\begin{figure}[H]
    \centering
    \includegraphics[scale=0.7]{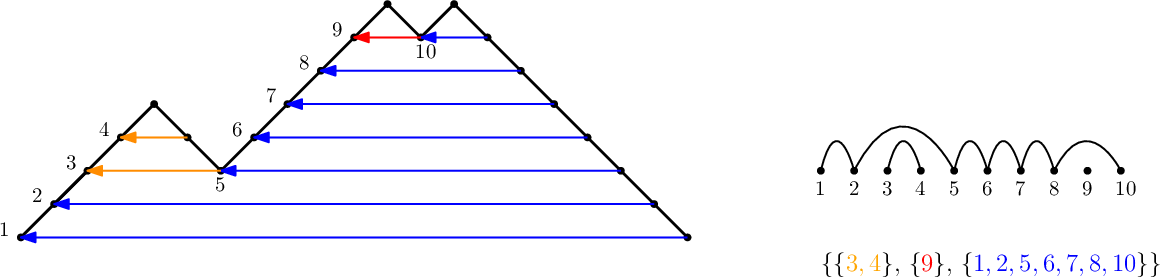}
    \caption{A bijection between a non-crossing partition and a Dyck path.}
    \label{NCfig2}
\end{figure}

Following the bijection $\Phi$, we can consider the following characterization for the family of $d$-Dyck paths in terms of partitions. 
We denote by $\NC_d(n)$ the set 
\begin{align}
\label{ncdeq}
   \{B_1/B_2/\cdots /B_k  \in \NC(n): |([n]\setminus [\texttt{a}_{i+1}])\cap B_i|\geq d,\,   \texttt{a}_{i+1}= \max (B_{i+1})  \text{ for }1<i<k \}. 
\end{align}
That is, a partition $\pi = B_1/B_2/\cdots /B_k$ belongs to $ \NC_d(n)$ if and only if for all $1<i<k $, there are at least $d$ elements in the block $B_i$ bigger than the maximum element in $B_{i+1}$.
 
The \emph{reverse} of a partition $\pi=B_1/B_2/\cdots/B_k$ of $[n]$,  is the partition $\pi^R=B_1^{\prime}/B_2^{\prime}/\cdots/B_k^{\prime}$, where $B_i^{\prime}=n+1-B_i:=\{n+1-\ell: \ell\in B_i\}$. It is clear that this operation is a bijection. 

We now construct a bijection to show that $|\NC_2(13)|=r_2(13)$. For this goal, we consider the $2$-Dyck path of semi-length $13$, given on left-hand side of   
Figure \ref{bij11},  $P=U^2D^2U^3DU^4DU^3DUD^8$; and the path $P^{\prime}$ obtained from $P$ ---interchanging the roles of  
 $U$ and $D$---, (see the right-hand side of Figure \ref{bij11}). That is, $P^{\prime}=D^2U^2D^3UD^4UD^3UDU^8$. The path $P^{\prime}$ can be seen 
 as a reflection of $P$ with respect to the $(x=2n)$-axis. The valleys vector of $P$ is  $\nu=(0, 2, 5, 7)$ and the valleys vector of $P^{\prime}$ is 
 $\nu^{\prime}=(7,5,2,0)$. Observe that the valleys vector of $P^{\prime}$ satisfy that  $\nu_{i+1}^{\prime}-\nu_{i}^{\prime}\leq-2$.  Applying the bijection 
 $\Phi$ to $P^{\prime}$, we obtain that 
$$\Phi(P^{\prime})=\{\{1, 2, 11\}, \{3, 4, 5, 10\}, \{6, 7, 9\}, \{8\}, \{12, 13\}\},$$
and
$$\Phi(P^{\prime})^R=\{\{1, 2\}, \{3, 12, 13\}, \{4, 9, 10, 11\}, \{5, 7, 8\},  \{6\}\}\in \NC_2(13).$$
A graph representation of $\Phi(P^{\prime})^R$ is depicted in Figure \ref{bij2}.
\begin{figure}[H]
    \centering
    \includegraphics[scale=0.45]{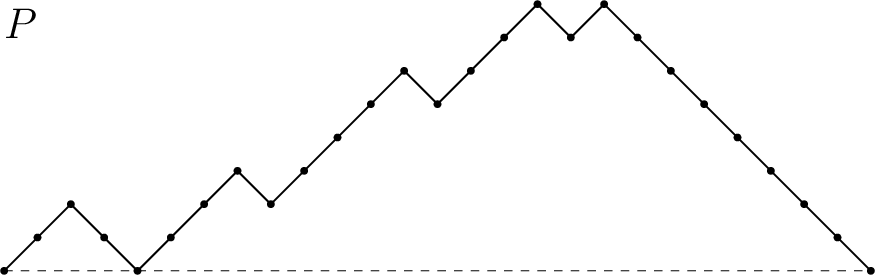} \hspace{0.4cm}
        \includegraphics[scale=0.5]{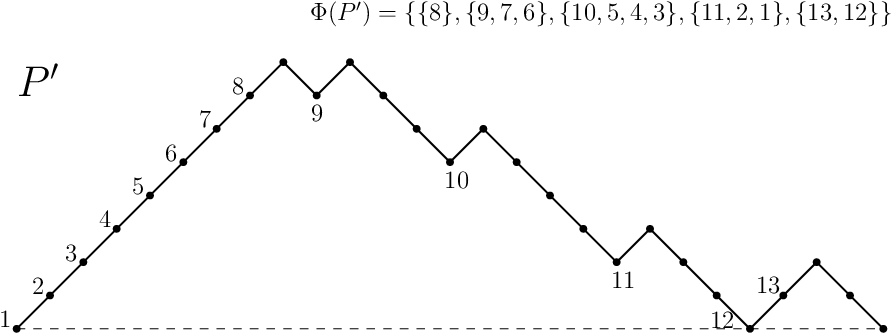}
    \caption{A Bijection between a non-crossing partitions and a Dyck path.}
    \label{bij11}
\end{figure}

\begin{figure}[H]
    \centering
        \includegraphics[scale=0.6]{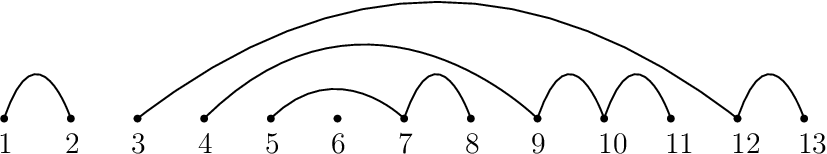}
    \caption{Graphical representation of $\Phi(P^{\prime})^R$.}
    \label{bij2}
\end{figure}

We now give a general statement of the previous example. Thus, 
this theorem gives a bijection between the set $\NC_d(n)$ and the set of $d$-Dyck paths of length $2n$. 

\begin{theorem}\label{dDyckPart}
If $d>0$ and $n\geq 0$, then the family of $d$-Dyck paths of length $2n$ and $\NC_d(n)$ are bijectively related. Furthermore,

$$|\NC_d(n)|=\sum_{k=0}^{\lfloor\frac{n+d-2}{d}\rfloor}\binom{n-(d-1)(k-1)}{2k}.$$
\end{theorem}

\begin{proof} Let $P$ be a $d$-Dyck path represented by a word over the alphabet $\{U,D\}$.  Let $P^{\prime}$ be the path traversed backwards,  
that means, exchange the $U'$s for $D's$ and vice versa and reverse the string. So the path $P^{\prime}$ is the reflection with respect to the  
$(x=2n)$-axis.  From this transformation, we can see that $P^{\prime}$ has the property that the valleys heights satisfy that $\nu_{i+1}-\nu_{i}\leq-d$.

We now apply the bijection $\Phi$ to $P^{\prime}$. Every valley in $P^{\prime}$ gives rise to a new block (containing its label) and each block has at least $d$ 
labels less than the label of the valley present in that block. Note that the first block of the partition does not follow this rule. This gives that the number of blocks 
in the partition is equal to the number of valleys plus one (the first block). Now applying the reverse to $\Phi(P^{\prime})$ we have that the condition of being  
\textit{smaller} becomes being \textit{larger}, that is $\Phi(P^{\prime})^R\in\NC_d(n)$. This gives  the characterization.
\end{proof}

The previous theorem gives rise to the question: what kind of interesting results do we obtain dropping the ``non-crossing condition"? With this question in   
mind we introduce the \emph{restricted $d$-Bell numbers}: let 
$$\Pi _d(n)=\{\pi =B_1/B_2/\cdots /B_k  \in \Pi(n):\text{for }1<i<k, |([n]\setminus [a])\cap B_i|\geq d\},$$
where $a=\max (B_{i+1})$ with $d\geq 0$. Notice how $\NC _d(n)\subseteq \Pi _d(n)$, so the extension with respect to $\NC _d(n)$, defined in \eqref{ncdeq}, is just considering all partitions in $\Pi (n)$ instead of just the non-crossing ones in $\NC(n)$. We use $B_d(n)$ to denote the cardinality of $\Pi _d(n)$; an element  $\pi \in \Pi _d(n)$ is called   
\emph{a restricted $d$-partition of} $[n]$. For example,  the partition $\pi=\{\{\{1\}, \{2, 10, 11\}, \{3, 8, 9\}, \{4, 6, 7\}, \{5\}\}\}$ is a restricted $2$-partition of the set  
$[11]$. Notice that for $d=0$ we recover the Bell numbers $B(n)$.

Theorem \ref{dBell} gives an answer to our question.  Thus,  this theorem gives a recurrence relation to calculate the number of restricted $d$-Bell numbers. 

\begin{theorem}\label{dBell}
The restricted $d$-Bell number $B_d(n)$ satisfies the recurrence relation
 $$B_d(n)=\sum _{k=0}^{n-1}\binom{n-1}{k}B_d(k-d),$$
 with $B_d(n)=1$ for $n\leq 1$. 
Furthermore, if $d\longrightarrow \infty$, then $B_d(n)=2^{n-1}$.
\end{theorem}

\begin{proof}
Let $\pi=B_1/\cdots /B_k$ be a partition in $\Pi _d(n)$ and let $\pi^R=B^{\prime}_k/B^{\prime}_{k-1}/\cdots /B^{\prime}_2/B^{\prime}_1$ 
be the reverse of $\pi$. It is easy to see that $n\in B^{\prime}_1$. From the condition on the partition $\pi^R$, we have that for $i>1$,   
the block  $B^{\prime}_i$ has at least $d$ elements smaller than the minimum element in $B^{\prime}_{i+1}$. We select a $(n-1-k)$-set   
$X \subseteq [n-1]$ satisfying  the condition that $\{n\}\cup X$ is equal to $B^{\prime}_1$. Let $M_d$ be the set of $d$ minimal elements of  
$[n-1]\setminus X$. We now create a restricted $d$-partition of  $[n-1]\setminus (X\cup M_d)$ and attach $M_d$ to the smallest block.  
This procedure gives the the desired recursion. 

Finally,  if $d>n$, then we cannot have three or more blocks in our partition. If there are more than two blocks, then we need an infinite number 
of elements to be placed in the middle partition. Since ${n \brace 1}=1, {n\brace 2}=2^{n-1}-1$, adding them we get $B_d(n)=2^{n-1}$.
\end{proof}

The numbers $B_d(n)$ agree with the sequence A210545  in~\cite{sloane} shifting $d$ units in the row. In Table \ref{tab2} we show the first few values of the sequence $B_d(n)$.
\begin{table}[htp]
\centering
\begin{tabular}{l||lllllllllll}
  $d\backslash n$ & $0$ & $1$ & $2$ & $3$ & $4$ & $5$ & $6$ & $7$ & $8$ & $9$ & $10$ \\ \hline\hline
   $d=0$ &  1 & 1 & 2 & 5 & 15 & 52 & 203 & 877 & 4140 & 21147 & 115975\\
  $d=1$& 1 & 1 & 2 & 4 & 9 & 23 & 65 & 199 & 654 & 2296 & 8569 \\
  $d=2$& 1 & 1 & 2 & 4 & 8 & 17 & 40 & 104 & 291 & 857 & 2634  \\
  $d=3$& 1 & 1 & 2 & 4 & 8 & 16 & 33 & 73 & 177 & 467 & 1309 \\
  $d=4$&1 & 1 & 2 & 4 & 8 & 16 & 32 & 65 & 138 & 315 & 782 \\
\end{tabular}\\[6pt]
\caption{Values of $B_d(n)$ for $d=0,1,2,3,4$.}\label{tab2}
\end{table}

\textbf{Acknowledgments}.  
The first author was partially supported by The Citadel Foundation.
The second author was partially supported by the Facultad de Ciencias at Universidad Nacional de Colombia,
Project No.  53490.

\end{document}